\documentclass[11pt]{article}
\usepackage[a4paper,margin=1in]{geometry}
\usepackage[T1]{fontenc}
\usepackage{lmodern}
\usepackage{microtype}
\usepackage{amsmath,amssymb,amsthm,mathtools}
\usepackage{enumitem}
\usepackage{bm}
\usepackage{xcolor}
\usepackage{hyperref}
\hypersetup{colorlinks=true,linkcolor=blue,citecolor=blue,urlcolor=blue}

\title{Rational arrival processes with strictly positive densities need not be Markovian}
\author{Oscar Peralta}
\date{}

\newtheorem{theorem}{Theorem}
\newtheorem{lemma}[theorem]{Lemma}
\newtheorem{proposition}[theorem]{Proposition}

\newtheorem{conjecture}[theorem]{Conjecture}

\theoremstyle{definition}

\theoremstyle{remark}

\newcommand{\e}{\mathrm{e}}

\newcommand{\bzero}{\bm{0}}
\newcommand{\R}{\mathbb{R}}
\newcommand{\Q}{\mathbb{Q}}

\begin{document}
\maketitle

\begin{abstract}
Telek~\cite{Telek2022} asked whether a rational arrival process (RAP), specified by matrices $\bm{G}_0$ and $\bm{G}_1$ and an initial row vector $\bm{\nu}$, with strictly positive joint densities and a unique dominant real eigenvalue of $\bm{G}_0$ must admit an equivalent Markovian arrival process (MAP). A counterexample of order $3$ is given, showing the answer is no, and that the conjecture fails even under the stronger condition of exact normalisation $(\bm{G}_0+\bm{G}_1)\bm{1}=\bm{0}$. The construction combines a strictly positive exponential baseline with a two-dimensional correction driven by an irrational rotation. Strict positivity of all joint densities follows from the continuous-time damping of the correction block; the obstruction to MAP realisability comes from the poles of the boundary generating function at $\e^{\pm i\varphi}$, which cannot be peripheral eigenvalues of any finite nonnegative matrix when $\varphi/\pi\notin\Q$.
\end{abstract}

\section{Introduction}

Phase-type (PH) distributions and matrix-exponential (ME) distributions describe the same kind of object, a distribution on $\R_+$ given by a vector--matrix pair $(\bm{\nu},\bm{G}_0)$, but with different admissibility conditions. A PH representation comes from a finite Markov chain and is therefore sign constrained. An ME representation keeps the same formula for the density,
\[
f(t)=-\bm{\nu} \e^{\bm{G}_0 t} \bm{G}_0\bm{1},
\]
but drops the Markovian interpretation; the question of giving such a representation a probabilistic interpretation goes back to Cox~\cite{cox1955use} and has been revisited more recently~\cite{peralta2023markov}.
The relation between the two classes has been studied for a long time. O'Cinneide's theorem shows that strict positivity together with a dominant-eigenvalue condition is enough to recover a PH representation from an ME one~\cite{OCinneide1990}. Later work made that picture more constructive~\cite{MocanuCommault1999,CommaultMocanu2003,HorvathTelek2015}.

The same distinction appears for point processes.
A finite-dimensional MAP is described by a triple $(\bm{\nu},\bm{G}_0,\bm{G}_1)$ with the usual Markovian sign constraints: $\bm{\nu}\ge\bzero$, $\bm{G}_1\ge\bzero$, the off-diagonal entries of $\bm{G}_0$ are nonnegative, and $\bm{G}_0\bm{1}\le\bzero$.
A RAP, introduced by Asmussen and Bladt~\cite{AsmussenBladt1999}, keeps the same joint-density formula,
\begin{equation}\label{eq:joint-density}
f_k(t_1,\dots,t_k)
\;=
\bm{\nu}\,\e^{\bm{G}_0 t_1}\bm{G}_1\,\e^{\bm{G}_0 t_2}\bm{G}_1\cdots \e^{\bm{G}_0 t_k}\bm{G}_1\bm{1},
\qquad t_1,\dots,t_k\ge 0,
\end{equation}
but allows arbitrary real entries as long as the resulting densities are nonnegative for every $k$.
The relationship between these two classes was surveyed in~\cite{asmussen2022ph}; RAPs have also been used as modelling tools in fluid queues~\cite{bean2022rap} and quasi-birth-and-death processes~\cite{bean2010quasi}.
Every MAP is therefore a RAP, but not every RAP is visibly Markovian.
The real question is whether suitable regularity assumptions force the hidden Markov structure back in.

Telek formulated exactly that question in~\cite{Telek2022}. The conjecture is the point-process analogue of the classical ME/PH characterisation: if the joint densities are strictly positive and if $\bm{G}_0$ has a unique real eigenvalue of maximal real part, does one automatically get a finite MAP representation? In the ME/PH setting the dominant-eigenvalue condition rules out the oscillatory behaviour that prevents a positive function from being PH. One might therefore expect the same mechanism to work here as well. The present note shows that it does not. More than that, it still does not work after one strengthens Telek's admissibility requirements to the normalised setting of a non-terminating point-process model.

The reason the ME/PH proof does not carry over is structural. In the one-matrix problem (that is, the characterisation of PH distributions), the density $f(t)=-\bm{\nu}\e^{\bm{G}_0t}\bm{G}_0\bm{1}$ involves only $\bm{G}_0$; the spectral hypotheses on $\bm{G}_0$ and the positivity of~$f$ together constrain the single matrix enough to force a Markovian realisation. In the two-matrix problem, the $k$-fold density~\eqref{eq:joint-density} involves iterated products $\e^{\bm{G}_0t_1}\bm{G}_1\e^{\bm{G}_0t_2}\bm{G}_1\cdots$, so the matrix $\bm{G}_1$ enters as an algebraically independent object. The hypotheses of Conjecture~\ref{conj:telek} constrain $\bm{G}_0$ directly (through the dominant-eigenvalue condition) and the pair $(\bm{G}_0,\bm{G}_1)$ jointly (through positivity of the densities), but they say almost nothing about $\bm{G}_1$ alone. The question is therefore whether positivity of all joint densities places enough indirect constraints on $\bm{G}_1$ to force nonnegativity. The answer turns out to be no: the continuous-time damping from $\bm{G}_0$ can mask bounded oscillatory behaviour in $\bm{G}_1$ that is fundamentally incompatible with any nonnegative realisation.

\begin{conjecture}[Telek~\cite{Telek2022}]\label{conj:telek}
Any RAP with strictly positive joint densities on $(0,\infty)^k$ for every $k\ge 1$ and with $\bm{G}_0$ having a unique real eigenvalue of maximal real part has a finite-dimensional MAP representation.
\end{conjecture}

Telek's note formulates the problem with the weaker condition
\[
\int_0^\infty\cdots \int_0^\infty f_k(t_1,\dots,t_k)
\,dt_1\cdots \,dt_k\le 1.
\]
For the present purpose it is more natural to impose the exact consistency condition
\begin{equation}\label{eq:marginal-consistency-intro}
\bm{G}_0\bm{1} + \bm{G}_1\bm{1} = \bzero,
\end{equation}
which is the analogue of the usual conservation-of-mass identity in a MAP. This condition is implicit in the finite-dimensional point-process framework of Asmussen and Bladt~\cite{AsmussenBladt1999}, though it is not stated explicitly there; we make it explicit here. Under~\eqref{eq:marginal-consistency-intro}, together with $\bm{\nu}\bm{1}=1$ and $\e^{\bm{G}_0t}\bm{1}\to\bzero$, each $f_k$ integrates to $1$ and the family $(f_k)_{k\ge 1}$ is projectively consistent under integration in the last variable. The example below is therefore not a defective or partially normalised object; it is a genuinely normalised RAP, and in particular it strengthens Telek's setup.

The counterexample we provide has order $3$. The first coordinate provides a strictly positive exponential baseline against which a signed correction can be added. The remaining two coordinates form the correction block, which decays faster in continuous time, with rate $2$. At each arrival, however, the correction block is multiplied by a rotation by angle $\varphi$. Because orthogonal rotations have bounded powers, the perturbation remains small and strict positivity is preserved. Choosing $\varphi/\pi$ irrational ensures that the successive rotation directions never repeat periodically, which is precisely what a finite nonnegative matrix cannot reproduce on its peripheral spectrum.

Our result is thus conceptually close to the PH-characterisation literature, but with the opposite conclusion. The classical results show that positivity together with the right spectral assumptions is enough to force a Markovian representation. Here the same philosophy breaks down one level higher: positivity of all joint densities and a simple dominant eigenvalue for $\bm{G}_0$ still do not prevent a non-Markovian oscillatory mode from hiding in the arrival mechanism.

The heuristic that leads to the counterexample can be described as follows. If a MAP representation $(\bm{\alpha},\bm{C}_0,\bm{C}_1)$ were to exist, the boundary values $a_k:=\bm{\nu}\bm{G}_1^k\bm{1}$ would have to equal $\bm{\alpha}\bm{C}_1^k\bm{1}$ for a nonnegative matrix $\bm{C}_1$. This is a scalar nonnegative realisation problem, and it carries a classical spectral obstruction: by the Perron--Frobenius theorem, every eigenvalue of a finite nonnegative matrix on its spectral circle must be a root of unity times the spectral radius. The simplest bounded oscillation that violates this constraint is a rotation by an irrational angle: the sequence $\e^{ij\varphi}$ with $\varphi/\pi\notin\Q$ is bounded (it lives on the unit circle) but never periodic, since by Weyl's equidistribution theorem its values are dense on the circle. The generating function of the resulting boundary sequence then has poles at $\e^{\pm i\varphi}$, which are not roots of unity. Once this obstruction is identified, the construction is almost forced: one needs a RAP whose matrix $\bm{G}_1$ encodes an irrational rotation in a $2\times2$ block, together with enough continuous-time damping in $\bm{G}_0$ to keep all joint densities strictly positive.

The rest of the paper is organised as follows.
Section~\ref{sec:construction} presents the counterexample in full and derives a closed form for the joint densities.
Section~\ref{sec:verification} verifies exact normalisation, strict positivity, and the dominant-eigenvalue condition.
Section~\ref{sec:boundary} studies the behaviour of the joint densities
at the origin and computes the generating function of the resulting sequence.
Section~\ref{sec:no-MAP} uses these results to prove that no finite-dimensional MAP representation can exist.

\section{The counterexample}\label{sec:construction}

Throughout, matrices and vectors are written in bold. We write $\bm{1}$ for the all-ones column vector of appropriate dimension, $\bm{1}_2=(1,1)^\top$ for the all-ones column vector in $\mathbb{R}^2$, $\bzero$ for the zero column vector of appropriate dimension, and $\bm{I}$, $\bm{I}_2$ for the identity matrix of appropriate and of size $2$, respectively.

The key idea is to encode inside the arrival matrix $\bm{G}_1$ a two-dimensional rotation acting on a correction part of the intensity vector; this produces the oscillatory terms in the boundary sequence. The rotation is introduced via the matrix
\begin{equation}\label{eq:irrational}
\varphi\in(0,\pi),
\qquad
\frac{\varphi}{\pi}\notin\mathbb{Q},
\end{equation}
and
\[
c=\cos\varphi,
\qquad
s=\sin\varphi,
\qquad
\bm{R}_\varphi=
\begin{pmatrix}
c & -s\\
s & c
\end{pmatrix}.
\]
At each arrival, the two correction coordinates of the intensity vector are transformed by $\bm{R}_\varphi$. After $j$ arrivals the correction component is multiplied by $\bm{R}_\varphi^j = \bm{R}_{j\varphi}$, so each arrival rotates the correction coordinates by one further angle $\varphi$. Since $\bm{R}_\varphi$ is orthogonal, its powers remain bounded. The condition $\varphi/\pi\notin\mathbb{Q}$ ensures that the successive directions never repeat: the sequence $\{\bm{R}_\varphi^j\}_{j\ge 0}$ is not eventually periodic.

Since $\bm{R}_\varphi$ has signed entries it cannot serve directly as a nonnegative arrival matrix. A third coordinate, the positive-intensity coordinate, is therefore added; its sole role is to contribute enough positive mass to keep the total arrival intensities nonnegative. Accordingly, we set up $\bm{G}_1$ in the block form
\begin{equation}\label{eq:G1-shape}
\bm{G}_1=
\begin{pmatrix}
b & \bzero^\top\\
\bm{u} & \bm{R}_\varphi
\end{pmatrix},
\qquad
\bm{u}=
\begin{pmatrix}
u_1\\
u_2
\end{pmatrix},
\end{equation}
where the bottom-right block $\bm{R}_\varphi$ acts on the correction coordinates, the scalar $b$ is the arrival rate of the positive-intensity coordinate, and the vector $\bm{u}$ encodes the coupling between that coordinate and the rotating block; all three are determined by the condition $(\bm{G}_0+\bm{G}_1)\bm{1}=\bzero$ below.

We take $\bm{G}_0$ to be diagonal, so that between arrivals each coordinate decays exponentially at its own rate with no further rotation or mixing. Choosing a faster decay rate for the correction coordinates keeps their contribution small:
\begin{equation}\label{eq:G0-choice}
\bm{G}_0=
\begin{pmatrix}
-1 & 0 & 0\\
0 & -2 & 0\\
0 & 0 & -2
\end{pmatrix}.
\end{equation}

The constants $b$, $u_1$, $u_2$ in~\eqref{eq:G1-shape} are now fixed by \eqref{eq:marginal-consistency-intro}, which gives
\begin{equation}\label{eq:u-def}
b=1,
\qquad
u_1=2-c+s,
\qquad
u_2=2-c-s.
\end{equation}
Therefore
\begin{equation}\label{eq:full-G1}
\bm{G}_1=
\begin{pmatrix}
1 & 0 & 0\\
2-c+s & c & -s\\
2-c-s & s & c
\end{pmatrix}.
\end{equation}

The initial distribution is
\begin{equation}\label{eq:full-nu}
\bm{\nu}=(1-\varepsilon,\,\varepsilon,\,0),
\end{equation}
where $\varepsilon\in(0,1)$ is a small parameter; its admissible range is determined in Proposition~\ref{prop:positivity}.

For the closed-form joint densities, write
\[
T_j=t_1+\cdots+t_j,
\qquad
T=T_k=t_1+\cdots+t_k.
\]
Because $\bm{G}_0$ is diagonal and $\bm{G}_1$ is block lower triangular, the product in~\eqref{eq:joint-density} can be computed explicitly.
\begin{lemma}\label{lem:density-formula}
For every $k\ge 1$ and every $t_1,\dots,t_k\ge 0$,
\begin{equation}\label{eq:density-formula}
f_k(t_1,\dots,t_k)
=
\e^{-T}
\left[
1-\varepsilon
+
\varepsilon\,(1,0)
\left(
\sum_{j=1}^{k-1}
\e^{-T_j}(2\bm{R}_\varphi^{j-1}-\bm{R}_\varphi^j)\bm{1}_2
+
2\e^{-T}\bm{R}_\varphi^{k-1}\bm{1}_2
\right)
\right],
\end{equation}
where $\bm{1}_2=(1,1)^\top$ is the all-ones vector in $\mathbb{R}^2$.
\end{lemma}

\begin{proof}
Throughout this proof, fix ${k_0}\ge 1$. For $i=1,\dots,{k_0}$, set $\bm{M}_i=\e^{\bm{G}_0t_i}\bm{G}_1$, so that
\[
f_{k_0}(t_1,\dots,t_{k_0})=\bm{\nu}\bm{M}_1\cdots \bm{M}_{k_0}\bm{1}.
\]
For $j=1,\dots,{k_0}$, write
\[
\bm{M}_j\cdots \bm{M}_{k_0}\bm{1}
=
\begin{pmatrix}
\e^{-(T-T_{j-1})}\\
\bm{v}_j
\end{pmatrix},
\qquad T_0:=0,
\]
with $\bm{v}_j\in\mathbb{R}^2$. Using the block forms
\[
\bm{G}_1=
\begin{pmatrix}
1 & \bzero^\top\\
\bm{u} & \bm{R}_\varphi
\end{pmatrix},
\qquad
\bm{u}=(2\bm{I}_2-\bm{R}_\varphi)\bm{1}_2,\qquad\mbox{and}\qquad
\e^{\bm{G}_0t_j}
=
\begin{pmatrix}
\e^{-t_j} & \bzero^\top\\
\bzero & \e^{-2t_j}\bm{I}_2
\end{pmatrix},
\]
we obtain
\[
\bm{M}_j
=
\begin{pmatrix}
\e^{-t_j} & \bzero^\top\\
\e^{-2t_j}\bm{u} & \e^{-2t_j}\bm{R}_\varphi
\end{pmatrix}.
\]
Hence
\[
\bm{v}_j
=
\e^{-2t_j}\bm{u}\,\e^{-(T-T_j)}
+
\e^{-2t_j}\bm{R}_\varphi\,\bm{v}_{j+1},
\qquad j=1,\dots,{k_0}-1.
\]
Also, since $\bm{G}_1\bm{1}=-\bm{G}_0\bm{1}= (1, 2, 2)^\intercal$, we have
\[
\bm{M}_{k_0}\bm{1}
=
\e^{\bm{G}_0t_{k_0}}\bm{G}_1\bm{1}
=
\begin{pmatrix}
\e^{-t_{k_0}}\\
2\e^{-2t_{k_0}}\bm{1}_2
\end{pmatrix},
\]
and therefore
\[
\bm{v}_{k_0}=2\e^{-2t_{k_0}}\bm{1}_2.
\]

Iterating the recursion gives
\[
\bm{v}_1
=
\sum_{j=1}^{{k_0}-1}
\e^{-2T_j}\bm{R}_\varphi^{j-1}\bm{u}\,\e^{-(T-T_j)}
+
2\e^{-2T}\bm{R}_\varphi^{{k_0}-1}\bm{1}_2.
\]
Substituting $\bm{u}=(2\bm{I}_2-\bm{R}_\varphi)\bm{1}_2$ and factoring out $\e^{-T}$ yields
\[
\bm{v}_1
=
\e^{-T}
\left(
\sum_{j=1}^{{k_0}-1}
\e^{-T_j}(2\bm{R}_\varphi^{j-1}-\bm{R}_\varphi^j)\bm{1}_2
+
2\e^{-T}\bm{R}_\varphi^{{k_0}-1}\bm{1}_2
\right).
\]
Finally, since
\[
\bm{\nu}=(1-\varepsilon,\varepsilon,0),
\]
we have
\[
f_{k_0}
=
(1-\varepsilon)\e^{-T}
+
\varepsilon\,(1,0)\bm{v}_1,
\]
which is exactly \eqref{eq:density-formula}.
\end{proof}

Formula~\eqref{eq:density-formula} separates the density into a positive leading term $(1-\varepsilon)\e^{-T}$ and a signed correction. In the correction, each summand carries an extra factor \(\e^{-T_j}\), so although the oscillation in \(\bm{R}_\varphi^j\) does not decay with \(j\), its contribution is damped by the elapsed time.

\section{Normalisation, positivity, and the spectral condition}\label{sec:verification}

We verify the three properties that qualify the example as a counterexample to Conjecture~\ref{conj:telek}.

\begin{proposition}\label{prop:normalisation}
For every $k\ge 2$,
\begin{equation}\label{eq:one-step-normalisation}
\int_0^{\infty} f_k(t_1,\dots,t_k)\,dt_k
=
 f_{k-1}(t_1,\dots,t_{k-1}),
\end{equation}
and
\begin{equation}\label{eq:full-normalisation}
\int_0^\infty\cdots \int_0^\infty f_k(t_1,\dots,t_k)
\,dt_1\cdots \,dt_k=1
\qquad\text{for all }k\ge 1.
\end{equation}
\end{proposition}

\begin{proof}
Since $(\bm{G}_0+\bm{G}_1)\bm{1}=\bzero$, we have
\[
\int_0^\infty \e^{\bm{G}_0t}\bm{G}_1\bm{1}\,dt
=
-\int_0^\infty \e^{\bm{G}_0t}\bm{G}_0\bm{1}\,dt
\;=\; -\left[\e^{\bm{G}_0t}\bm{1}\right]_0^\infty
=
\bm{1}.
\]
Substituting this into the last factor of \eqref{eq:joint-density} gives \eqref{eq:one-step-normalisation}. Iterating \eqref{eq:one-step-normalisation} reduces \eqref{eq:full-normalisation} to the case $k=1$, which gives $\int_0^\infty f_1(t_1)\,dt_1=\bm{\nu}\bm{1}=1$.
\end{proof}

\begin{proposition}\label{prop:positivity}
There is a constant $M=M(\varphi)$ such that, whenever
\[
0<\varepsilon<\frac{1}{M+1},
\]
we have
\[
f_k(t_1,\dots,t_k)>0
\qquad\text{for every }k\ge 1\text{ and every }t_1,\dots,t_k\ge 0.
\]
\end{proposition}
\begin{proof}
Fix $k_0\ge 1$. By \eqref{eq:density-formula},
\[
f_{k_0}(t_1,\dots,t_{k_0})
=
\e^{-T}
\left[
1-\varepsilon
+
\varepsilon
\left(
\sum_{j=1}^{{k_0}-1}
\e^{-T_j}(1,0)(2\bm{R}_\varphi^{j-1}-\bm{R}_\varphi^j)\bm{1}_2
+
2\e^{-T}(1,0)\bm{R}_\varphi^{{k_0}-1}\bm{1}_2
\right)
\right].
\]
Set
\[
w_j:=(1,0)(2\bm{R}_\varphi^{j-1}-\bm{R}_\varphi^j)\bm{1}_2
\quad (1\le j\le {k_0}-1),
\qquad
w_{k_0}:=2(1,0)\bm{R}_\varphi^{{k_0}-1}\bm{1}_2,
\]
and let
\[
W_m:=\sum_{j=1}^m w_j.
\]
A direct telescoping computation gives
\[
W_m
\;=\; (1,0)\left( 2\sum_{j=0}^{m-1}\bm{R}_\varphi^j - \sum_{j=1}^{m}\bm{R}_\varphi^j \right)\bm{1}_2
=
(1,0)\left(\bm{I}_2-\bm{R}_\varphi^m+\sum_{\ell=0}^{m-1}\bm{R}_\varphi^\ell\right)\bm{1}_2,
\qquad 1\le m\le {k_0}-1,
\]
and
\[
W_{k_0}
\;=\; W_{k_0-1} + 2(1,0)\bm{R}_\varphi^{k_0-1}\bm{1}_2
=
(1,0)\left(\bm{I}_2+\sum_{\ell=0}^{{k_0}-1}\bm{R}_\varphi^\ell\right)\bm{1}_2.
\]

Identify $\mathbb{R}^2$ with $\mathbb{C}$ via $(x,y)\leftrightarrow x+iy$, so that $\bm{R}_\varphi$ acts as multiplication by $e^{i\varphi}$. Any real polynomial $\bm{A}$ in $\bm{R}_\varphi$ corresponds to a complex scalar $z_{\bm{A}}$, and $(1,0)\bm{A}\bm{1}_2 = \operatorname{Re}(z_{\bm{A}}(1+i))$, so $|(1,0)\bm{A}\bm{1}_2|\le\sqrt{2}|z_{\bm{A}}|$. Specifically, for
\[
\bm{A}_m := \bm{I}_2-\bm{R}_\varphi^m+\sum_{\ell=0}^{m-1}\bm{R}_\varphi^\ell,
\quad\mbox{we have}\quad
z_{\bm{A}_m} = 1-e^{im\varphi}+\sum_{\ell=0}^{m-1}e^{i\ell\varphi} = 1-e^{im\varphi}+\frac{1-e^{im\varphi}}{1-e^{i\varphi}}.
\]
Since \(\varphi\in(0,\pi)\), we have \(e^{i\varphi}\neq 1\), so the denominator \(|1-e^{i\varphi}|=2\sin(\varphi/2)\) is strictly positive (dropping the absolute value since \(\sin(\varphi/2)>0\)). By the triangle inequality,
\[
|z_{\bm{A}_m}|
\le
|1-e^{im\varphi}|+\frac{|1-e^{im\varphi}|}{|1-e^{i\varphi}|}
\le
2+\frac{2}{2\sin(\varphi/2)}
=
2+\frac{1}{\sin(\varphi/2)}.
\]
Therefore,
\[
|W_m| \le \sqrt{2}|z_{\bm{A}_m}| \le \sqrt{2}\left(2+\frac{1}{\sin(\varphi/2)}\right).
\]
The same logic applies to \(W_{k_0}\). Letting \(\bm{A}_{k_0} := \bm{I}_2+\sum_{\ell=0}^{k_0-1}\bm{R}_\varphi^\ell\), its complex counterpart \(z_{\bm{A}_{k_0}}\) gives
\[
|W_{k_0}|
\le
\sqrt{2}|z_{\bm{A}_{k_0}}|
=
\sqrt{2}\left|1+\frac{1-e^{ik_0\varphi}}{1-e^{i\varphi}}\right|
\le
\sqrt{2}\left(1+\frac{1}{\sin(\varphi/2)}\right).
\]
Setting \(M:=\sqrt{2}\left(2+\frac{1}{\sin(\varphi/2)}\right)\), hence
\[
|W_m|\le M
\qquad\text{for all }1\le m\le k_0.
\]

Now write
\[
\sum_{j=1}^{k_0} e^{-T_j}w_j
=
\sum_{j=1}^{{k_0}-1}W_j\bigl(e^{-T_j}-e^{-T_{j+1}}\bigr)+W_{k_0}e^{-T_{k_0}}.
\]
Since \(0\le e^{-T_{k_0}}\le\cdots\le e^{-T_1}\le1\), it follows that
\[
\left|\sum_{j=1}^{k_0} e^{-T_j}w_j\right|
\le
M\sum_{j=1}^{{k_0}-1}(e^{-T_j}-e^{-T_{j+1}})+Me^{-T_{k_0}}
=
Me^{-T_1}
\le M.
\]
Therefore
\[
f_{k_0}(t_1,\dots,t_{k_0})\ge e^{-T}\bigl(1-\varepsilon-\varepsilon M\bigr),
\]
which is strictly positive as soon as \(\varepsilon<1/(M+1)\).
\end{proof}

In the remainder of the paper, fix once and for all an $\varepsilon$ satisfying $0<\varepsilon<1/(M+1)$. With this choice, the triple $(\bm{\nu},\bm{G}_0,\bm{G}_1)$ is indeed a RAP with strictly positive joint densities.

\begin{proposition}\label{prop:dominant}
The matrix $\bm{G}_0$ has a unique real eigenvalue of maximal real part, namely $-1$.
\end{proposition}

\begin{proof}
The spectrum of $\bm{G}_0$ is $\{-1,-2,-2\}$.
So $-1$ is simple and strictly dominates the other eigenvalues in real part.
\end{proof}

\section{Evaluation at the origin and its generating function}\label{sec:boundary}

Setting $t_1=\cdots=t_k=0$ in~\eqref{eq:joint-density} removes every factor $\e^{\bm{G}_0t_i}$ and leaves behind iterated powers of $\bm{G}_1$ alone; define
\begin{equation}\label{eq:shadow}
a_k:=\bm{\nu}\bm{G}_1^k\bm{1},
\qquad k\ge 1.
\end{equation}
If a MAP representation $(\bm{\alpha},\bm{C}_0,\bm{C}_1)$ exists, then $a_k=\bm{\alpha}\bm{C}_1^k\bm{1}$ for a nonnegative matrix $\bm{C}_1$, so the question of MAP realisability reduces, at the level of these values, to a nonnegative realisation problem for a scalar sequence, which is a classical problem with known spectral obstructions.

\begin{lemma}\label{lem:shadow}
For every $k\ge 1$,
\begin{equation}\label{eq:shadow-formula}
a_k
=
1+\varepsilon(1,0)\sum_{j=0}^{k-1}\bm{R}_\varphi^j\bm{1}_2.
\end{equation}
In particular, $(a_k)_{k\ge 1}$ is positive and bounded.
\end{lemma}

\begin{proof}
Set \(t_1=\cdots=t_k=0\) in \eqref{eq:density-formula}. Then \(T_j=0\) for all \(j\), and therefore
\[
a_k
=
1-\varepsilon
+
\varepsilon(1,0)\left(
\sum_{j=1}^{k-1}(2\bm{R}_\varphi^{j-1}-\bm{R}_\varphi^j)\bm{1}_2
+
2\bm{R}_\varphi^{k-1}\bm{1}_2
\right).
\]
The sum telescopes:
\[
\sum_{j=1}^{k-1}(2\bm{R}_\varphi^{j-1}-\bm{R}_\varphi^j)+2\bm{R}_\varphi^{k-1}
=
\bm{I}_2+\sum_{j=0}^{k-1}\bm{R}_\varphi^j.
\]
Hence
\[
a_k
=
1-\varepsilon
+
\varepsilon(1,0)\left(\bm{I}_2+\sum_{j=0}^{k-1}\bm{R}_\varphi^j\right)\bm{1}_2
=
1+\varepsilon(1,0)\sum_{j=0}^{k-1}\bm{R}_\varphi^j\bm{1}_2,
\]
which is \eqref{eq:shadow-formula}. The sequence is positive because $a_k = f_k(0,\dots,0) > 0$ by Proposition~\ref{prop:positivity}. For boundedness, note that under the identification of $\mathbb{R}^2$ with $\mathbb{C}$,
\[
(1,0)\sum_{j=0}^{k-1}\bm{R}_\varphi^j\bm{1}_2
\quad\mbox{is identified with}\quad
\operatorname{Re}\sum_{j=0}^{k-1}e^{ij\varphi}(1+i)
=
\operatorname{Re}\left(\frac{1-e^{ik\varphi}}{1-e^{i\varphi}}(1+i)\right),
\]
whose modulus is bounded by
\[
\left|\frac{1-e^{ik\varphi}}{1-e^{i\varphi}}(1+i)\right|
=
\frac{|1-e^{ik\varphi}|}{|1-e^{i\varphi}|}\sqrt{2}
\le
\frac{2}{2\sin(\varphi/2)}\sqrt{2}
=
\frac{\sqrt{2}}{\sin(\varphi/2)},
\]
uniformly in $k$. Hence $(a_k)_{k\ge 1}$ is bounded.
\end{proof}

\begin{lemma}\label{lem:gf}
The generating function
\[
A(z):=\sum_{k=1}^{\infty} a_k z^k
\]
is given by
\begin{equation}\label{eq:gf-explicit}
A(z)
=
\frac{z}{1-z}
\left(
1+\varepsilon\frac{1-z(\cos\varphi+\sin\varphi)}{1-2z\cos\varphi+z^2}
\right).
\end{equation}
Its only possible poles are $z=1$ and $z=\e^{\pm i\varphi}$, and the poles at $z=\e^{\pm i\varphi}$ do not cancel.
\end{lemma}

\begin{proof}
By \eqref{eq:shadow-formula},
\begin{align*}
A(z)
&=
\sum_{k=1}^\infty z^k
+
\varepsilon(1,0)\sum_{k=1}^\infty z^k\sum_{j=0}^{k-1}\bm{R}_\varphi^j\bm{1}_2\\
&=
\frac{z}{1-z}
+
\frac{\varepsilon z}{1-z}\,(1,0)\sum_{j=0}^\infty (z\bm{R}_\varphi)^j\bm{1}_2\\
&=
\frac{z}{1-z}
+
\frac{\varepsilon z}{1-z}\,(1,0)(\bm{I}_2-z\bm{R}_\varphi)^{-1}\bm{1}_2,
\end{align*}
where the interchange of summation in the second step uses $|z|<1$, and the resulting 
rational expression~\eqref{eq:gf-explicit} is understood as the analytic continuation of $A$ 
to $\mathbb{C}$ minus its poles. A direct calculation gives
\[
(1,0)(\bm{I}_2-z\bm{R}_\varphi)^{-1}\bm{1}_2
\;=\; \frac{1}{1-2z\cos\varphi+z^2}(1,0)\begin{pmatrix} 1-z\cos\varphi & -z\sin\varphi \\ z\sin\varphi & 1-z\cos\varphi \end{pmatrix}\bm{1}_2
=
\frac{1-z(\cos\varphi+\sin\varphi)}{1-2z\cos\varphi+z^2},
\]
which proves \eqref{eq:gf-explicit}.

The denominator factors as
\[
1-2z\cos\varphi+z^2=(1-ze^{i\varphi})(1-ze^{-i\varphi}),
\]
so, besides the possible pole at $z=1$, the only other possible poles are at \(z=e^{\pm i\varphi}\). Since we have fixed \(\varepsilon>0\), these poles can disappear only if the numerator $1-z(\cos\varphi+\sin\varphi)$ vanishes there. At $z=e^{i\varphi}$, this would require
\[
1=e^{i\varphi}(\cos\varphi+\sin\varphi).
\]
Taking imaginary parts gives
\[
0=(\cos\varphi+\sin\varphi)\sin\varphi.
\]
Because \(\varphi\in(0,\pi)\), we have \(\sin\varphi>0\), hence \(\cos\varphi+\sin\varphi=0\). Substituting this back into the previous equation yields \(1=0\), a contradiction. Thus the numerator does not vanish at \(z=e^{i\varphi}\). The case \(z=e^{-i\varphi}\) is identical.
\end{proof}

\section{Non-existence of a MAP representation}\label{sec:no-MAP}

The argument uses only the sequence $(a_k)$ and the following proposition, whose proof rests on Perron--Frobenius theory for nonnegative matrices, applied blockwise via the Frobenius normal form.

An eigenvalue of a matrix is called peripheral if its modulus equals the spectral radius. For an irreducible nonnegative matrix, the Perron--Frobenius theorem implies that the peripheral eigenvalues are always roots of unity times the spectral radius, meaning that its discrete-time evolution can permute mass among states cyclically but only with a finite period. For a reducible matrix this applies to each irreducible diagonal block separately. In contrast, a real matrix with no sign constraints can encode a rotation by an arbitrary angle, producing a bounded oscillation with irrational frequency. The RAP framework allows $\bm{G}_1$ to carry an aperiodic rotation, while the MAP framework forces $\bm{C}_1\ge 0$, restricting its peripheral spectrum to finitely many equally-spaced phases.

\begin{proposition}\label{prop:matrix-obstruction}
Let $\bm{C}\ge 0$ be a finite square matrix and $\bm{\alpha}\ge\bzero$ a row vector, and suppose that
\[
b_k=\bm{\alpha}\bm{C}^k\bm{1} \qquad (k\ge 1)
\]
is a bounded sequence. Let $B(z):=\sum_{k=1}^\infty b_k z^k$ and suppose that $B$ has a pole at some $z_0$ with $|z_0|=1$. Then:
\begin{enumerate}[label=(\roman*)]
\item every pole of $B$ is the reciprocal of a nonzero eigenvalue of $\bm{C}$;
\item $\rho(\bm{C})=1$ and $\bm{C}$ has an eigenvalue $\lambda=z_0^{-1}$ with $|\lambda|=1$;
\item every such eigenvalue $\lambda$ is a root of unity.
\end{enumerate}
\end{proposition}

\begin{proof}
By discarding states unreachable from the support of $\bm{\alpha}$, we may assume without loss of generality that all states of $\bm{C}$ are reachable; this does not affect $(b_k)$ or $B$.

The identity $B(z)=\bm{\alpha}z\bm{C}(\bm{I}-z\bm{C})^{-1}\bm{1}$, valid first for $|z|<1/\rho(\bm{C})$ and then by analytic continuation, shows that every pole of $B$ is the reciprocal of a nonzero eigenvalue of $\bm{C}$ (since the poles correspond to roots of $\det(\bm{I}-z\bm{C})$). If $B$ has a pole at $z_0$ with $|z_0|=1$, then $\bm{C}$ has an eigenvalue $\lambda=z_0^{-1}$ with $|\lambda|=1$, so $\rho(\bm{C})\ge 1$.

We show $\rho(\bm{C})\le 1$. Assume for contradiction that $\rho(\bm{C})>1$. Put $\bm{C}$ into Frobenius normal form and choose an irreducible diagonal block $\bm{D}$ with spectral radius $\rho(\bm{C})$. Since all states are reachable from the support of $\bm{\alpha}$, there is some $m$ such that $\bm{\beta}:=(\bm{\alpha}\bm{C}^m)_{\bm{D}}$ is a nonzero nonnegative row vector. By  \cite[Theorem~8.4.4(c)]{HornJohnson2013}, $\bm{D}$ has a positive right eigenvector $\bm{q}$ with $\bm{D}\bm{q}=\rho(\bm{C})\bm{q}$ with entries that sum $1$; in particular $\bm{q}\le\bm{1}_{\bm{D}}$ entrywise. Since $\bm{C}\ge 0$ is block upper-triangular in Frobenius normal form, $(\bm{C}^n)_{\bm{D}\bm{D}}=\bm{D}^n$ and $(\bm{\alpha}\bm{C}^{m+n})_{\bm{D}}\ge\bm{\beta}\bm{D}^n$. Therefore
\[
b_{m+n} = \bm{\alpha}\bm{C}^{m+n}\bm{1} \ge \bm{\beta}\bm{D}^n\bm{1}_{\bm{D}} \ge \bm{\beta}\bm{D}^n\bm{q} = (\bm{\beta}\bm{q})\rho(\bm{C})^n.
\]
Since $\bm{\beta}\bm{q}>0$, the right-hand side is unbounded, contradicting the boundedness of $(b_k)$. Thus $\rho(\bm{C})=1$.

Finally, let $\lambda\in\sigma(\bm{C})$ satisfy $|\lambda|=1=\rho(\bm{C})$. By Frobenius normal form, $\lambda$ belongs to some irreducible diagonal block $\bm{D}$ with $\rho(\bm{D})=\rho(\bm{C})=1$, and Perron--Frobenius implies that every eigenvalue of $\bm{D}$ on the unit circle is of the form
\[
\e^{2\pi ir/h}, \qquad 0\le r\le h-1,
\]
for some integer $h\ge 1$~\cite[Corollary~8.4.6(c)]{HornJohnson2013}. Thus $\lambda$ is a root of unity.
\end{proof}

\begin{theorem}\label{thm:no-MAP}
For the fixed choice of $\varepsilon$ satisfying Proposition~\ref{prop:positivity}, the RAP defined by~\eqref{eq:full-nu}, \eqref{eq:full-G1}, and~\eqref{eq:G0-choice} has no equivalent finite-dimensional MAP representation.
\end{theorem}

\begin{proof}
Assume that an equivalent finite-dimensional MAP representation $(\bm{\alpha},\bm{C}_0,\bm{C}_1)$ exists, so in particular $\bm{\alpha}\ge\bzero$ and $\bm{C}_1\ge 0$. Since both matrix-exponential density formulas are continuous on $[0,\infty)^k$ and agree on $(0,\infty)^k$, they also agree at $t_1=\dots=t_k=0$. Therefore setting $t_1=\dots=t_k=0$ in both formulas gives $b_k:=\bm{\alpha}\bm{C}_1^k\bm{1}$ satisfying $b_k=a_k$ for all $k\ge 1$, so $B(z)=A(z)$. By Proposition~\ref{prop:matrix-obstruction}, every pole of $B$ is the reciprocal of an eigenvalue of $\bm{C}_1$.

By Lemma~\ref{lem:gf}, $A(z)$ has poles at $z=\e^{\pm i\varphi}$, hence so does $B(z)$. Therefore $\bm{C}_1$ has eigenvalues $\lambda_\pm=\e^{\mp i\varphi}$. Since $|\lambda_\pm|=1$, Proposition~\ref{prop:matrix-obstruction} implies that $\lambda_\pm$ must be roots of unity. But $\e^{\pm i\varphi}$ are not roots of unity because $\varphi/\pi\notin\mathbb{Q}$. This contradiction proves that no equivalent finite-dimensional MAP representation exists.
\end{proof}

The example shows that the hypotheses sufficient for the one-matrix ME/PH problem do not carry over to the two-matrix RAP/MAP problem: positivity of all joint densities and a simple dominant eigenvalue for $\bm{G}_0$ do not control the peripheral behaviour of the arrival matrix, since the irrational rotation produces boundedness without periodicity, the generating function inherits the irrational phases as poles, and Perron--Frobenius forces any nonnegative matrix to have only root-of-unity phases on its spectral circle.

The characterisation of MAPs within the class of RAPs is likely a harder problem than the already difficult characterisation of PH distributions within ME distributions, which was open for several decades before being resolved by O'Cinneide~\cite{OCinneide1990}. In the one-matrix problem, all spectral information is concentrated in $\bm{G}_0$, and the dominant-eigenvalue condition together with positivity suffices. In the two-matrix problem, $\bm{G}_1$ carries independent spectral information that positivity of the joint densities does not appear to control. Our example suggests that any corrected version of Telek's conjecture must impose conditions on the peripheral spectrum of $\bm{G}_1$, but what the precise conditions should be remains unclear to us. We hope the obstruction identified here, irrational rotation as a mechanism that is compatible with all positivity requirements yet incompatible with any nonnegative realisation, serves as a useful reference point for future work on this problem.

\bibliographystyle{plain}
\bibliography{RAPnotMAP2_refs}

\end{document}